\documentclass{amsart}
\usepackage{amsmath, enumerate}
\usepackage{algorithm}
\usepackage{algpseudocode}
\usepackage{amsfonts}
\usepackage{amssymb}
\usepackage{amsthm}
\usepackage{color}
\usepackage{caption}
\usepackage{amsrefs}
\usepackage{graphicx}
\usepackage{url}

\newcommand{\Z}{\mathbb{Z}}
\newcommand{\gam}{\Gamma}

\newcommand{\R}{\mathbb{R}}
\newcommand{\bF}{\mathbb{F}}

\newcommand{\G}{\Gamma}
\newcommand{\B}{\mathcal{B}}
\newcommand{\GB}{\gam_{\B}}

\newcommand{\T}{\mathcal{T}}

\newcommand{\haf}{H^*(A(\gam),\bF)}

\newcommand{\bp}{\begin{problem}}

\newcommand{\ep}{\end{problem}}

\newcommand{\ba}{\begin{answer}}

\newcommand{\ea}{\end{answer}}

\newcommand{\ben}{\renewcommand{\theenumi}{\alph{enumi}}

\renewcommand{\labelenumi}{(\theenumi)}\begin{enumerate}}

\newcommand{\een}{\end{enumerate}}

\newcommand{\sgn}{\mathrm{sgn}}

\newtheorem{defin}{Definition}[section]
\newtheorem{thm}[defin]{Theorem}
\newtheorem{cor}[defin]{Corollary}

\newtheorem{lem}[defin]{Lemma}
\newtheorem{prop}[defin]{Proposition}
\newtheorem{quest}[defin]{Question}

\newtheorem{ex}[defin]{Example}

\DeclareCaptionType{result}[Table][List of Results]

\title[The cohomology basis graph]{Right-angled Artin groups and the cohomology basis graph}
\begin{document}
\date{\today}

\author[R. Flores]{Ram\'{o}n Flores}
\address{Ram\'{o}n Flores, Department of Geometry and Topology, University of Seville, Spain}
\email{ramonjflores@us.es}

\author[D. Kahrobaei]{Delaram Kahrobaei}
\address{Delaram Kahrobaei, Departments of Mathematics and Computer Science, Queens College, CUNY, Department of Computer Science, University of York, UK,
Department of Computer Science and Engineeing, New York University, Tandon School of Engineering, The Initiative for the Theoretical Sciences, CUNY Graduate Center}
\email{delaram.kahrobaei@qc.cuny.edu, dk2572@nyu.edu, delaram.kahrobaei@york.ac.uk}

\author[T. Koberda]{Thomas Koberda}
\address{Thomas Koberda, Department of Mathematics, University of Virginia, Charlottesville, VA 22904}
\email{thomas.koberda@gmail.com}

\author[C. Le Coz]{Corentin Le Coz}
\address{Corentin Le Coz, Department of Mathematics: Algebra and Geometry (WE01), Universiteit Gent, Belgium}
\email{corentinlecoz@outlook.com}

\begin{abstract}
Let $\Gamma$ be a finite graph and let $A(\Gamma)$ be the corresponding right-angled Artin group. From an arbitrary basis $\mathcal B$ of $H^1(A(\Gamma),\mathbb F)$ over an arbitrary field, we construct a natural graph $\Gamma_{\mathcal B}$ from the cup product,
called the \emph{cohomology basis graph}. We show that $\Gamma_{\mathcal B}$ always contains $\Gamma$ as a subgraph. This provides
an effective way to reconstruct the defining graph $\Gamma$ from the cohomology of $A(\Gamma)$,
to characterize the planarity of the defining graph from the algebra of $A(\Gamma)$,
and to recover many other natural graph-theoretic invariants.
We also investigate the behavior
of the cohomology basis graph under passage to elementary subminors, and show that it is not well-behaved under edge contraction.

\end{abstract}

\maketitle

\section{Introduction}
\label{Intro}

This paper forms part of a program to study the relationship between the algebraic structure of right-angled Artin groups and the combinatorial
structure of graphs, and specifically
how one can extract combinatorial properties of a graph $\gam$ from an abstract group $G$ which is isomorphic to $A(\gam)$. The methods
of this paper investigate the interplay between group theory, linear algebra, algebraic topology, combinatorics, and commutative algebra
which arise in the study of graphs and right-angled Artin groups.
The reader is
directed to~\cite{koberda21survey} for background and commentary.

Translational tools between combinatorial properties and algebraic properties are interesting from a purely theoretical point of view, and
they also arise in applied contexts such as group based cryptography (cf.~\cite{FKK2019}, for instance). Computational tractability motivates
the present work to a high degree.

Good characterizations of many combinatorial properties of graphs via the algebraic structure of right-angled Artin groups have been obtained
by various authors.
For instance: being a nontrivial join~\cite{Servatius1989},
being disconnected~\cite{BradyMeier01}, containing a square~\cite{Kambites09,KK2013gt}, being a co-graph~\cite{KK2013gt,KK2018jt},
being a finite tree or complete bipartite graph~\cite{HS2007}, admitting a nontrivial automorphism~\cite{FKK2019}, being $k$--colorable~
\cite{FKK2020a},
fitting in a sequence of expanders~\cite{FKK2020exp}, admitting a Hamiltonian path or cycle~\cite{FKK-hamilton}, and being
(outer)planar~\cite{gheorghiu23}. In this paper, we present a new perspective on characterizing (outer)planarity of the underlying
graph, through the cohomology of the right-angled Artin group, which makes use of the properties of the Colin de Verdi\`ere invariant \cite{CdV1990} and permits
to describe other graph properties in terms of groups, as for example being a linear forest or being linklessly embeddable in $\R^3$.
 Moreover, our method also characterizes when the complement of the graph has some of these properties. The key and more difficult point of our argument is Theorem \ref{thm:main}, which establishes a certain embedding of graphs.


Let us now give a more precise setup. Let $\gam$ be a finite \emph{simplicial} graph.
That is, $\gam$ is undirected, and its geometric realization is a $1$--dimensional simplicial complex.
Such graphs are also sometimes called \emph{simple}.
This paper focusses on the problem of extracting combinatorial information about
$\Gamma$ from the associated right-angled Artin group
\[A(\gam)=\langle V(\gam)\mid [v,w]=1,\, \{v,w\}\in E(\gam)\rangle.\]

Let $\bF$ be an arbitrary field. We consider the cohomology ring $H^*(A(\gam),\bF)$. It is well-known that $\haf$ can be recovered
from an arbitrary presentation of $A(\gam)$; see Subsection~\ref{ss:biautomatic} below. Moreover, $A(\gam)$ is \emph{$1$--formal}, meaning
that $\haf$, and in particular the restriction of the cup product to the first degree, completely determines $\gam$. The fact that the
right-angled Artin group determines the underlying graph up to isomorphism is obtained as the main
result of~\cite{Droms87}, and a related rigidity result was established by~\cite{Sabalka2009}. An alternative proof
that the cohomology ring of a right-angled Artin group determines the underlying graph is given as Theorem 6.4 in~\cite{KoberdaGAFA};
cf.~Theorem 15.2.6 of~\cite{koberda21survey}.

\subsection{The cohomology basis graph and the defining graph}
We will be interested in effective ways of reconstructing $\gam$ from $A(\gam)$, especially
through $\haf$. For this, let $\mathcal B$ be an arbitrary
basis of $H^1(A(\gam),\bF)$. The \emph{cohomology basis graph} associated to $\GB$ is the graph whose vertices are elements of $\B$,
and whose edge relation is given by having a nontrivial cup product.

The main result of this paper is the following. We give an algebraic topology--combinatorics version, though there are many other equivalent
formulations; see Section~\ref{sect:characterizations} below. Here and in what follows, we do not require \emph{subgraphs} to be full.

\begin{thm}\label{thm:main}
Let $\gam$ be a finite simplicial graph and let $\B$ be an arbitrary basis for $H^1(A(\gam),\bF)$. Then $\gam$ is a subgraph of
$\GB$.
\end{thm}

The importance of this result for us is two--fold: first, it is the cornerstone on which
almost all results in the present paper are based, and
 permits the use of minor-monotonicity of the Colin de Verdi\`ere invariant to provide group-theoretic characterizations of many graph properties; cf.~Theorem \ref{cor:planar}. Second, it
has equivalent formulations in other contexts (such as in commutative algebra), wherein the corresponding results were
previously unknown.

In the course of the proof, we develop a novel perspective on
the computation of the determinant of an invertible matrix. This method greatly generalizes the approach used in
\cite{FKK-hamilton}, and highlights the role of certain graphs that arise naturally from the structure
of the minors of the matrix; cf.~Section \ref{sect:null-connected}.
This is where the main difficulty in establishing Theorem~\ref{thm:main}
lies: the main result implicitly finds a bijection between the vertices of $\gam$ and the basis $\B$ which has good algebraic properties,
 though in general no canonical bijection exists.


Since it is easy to see that $\gam$ and $\GB$ have the same number of vertices,
Theorem~\ref{thm:main} says that $\gam$ can be obtained from $\GB$ by deleting edges:

\begin{cor}\label{cor:minimal-edge}
For a finite simplicial graph $\gam$, we have that $\gam\cong\GB$ for any basis $\B$ of $H^1(A(\gam),\bF)$ which minimizes the number of
edges in $\GB$.
\end{cor}

From Corollary~\ref{cor:minimal-edge}, one can give an \emph{a priori} bound on
the complexity of reconstructing $\gam$ from $\haf$, since one can apply
the corollary to a field with two elements, over which there are only finitely many bases.

Recall that to every finite graph $\gam$, we may associate the \emph{Colin de Verdi\`ere invariant} $\mu(\gam)$, which is a natural
number that characterizes disconnected graphs, forests, outerplanar graphs, planar graphs, and many other classes of graphs.
The reader may find the definition and basic properties of $\mu(\gam)$ in Subsection~\ref{ss:cdv} below.


\begin{thm}\label{thm:cdv}
For a finite simplicial graph $\gam$ and natural number $k$, we have that $\mu(\gam)\leq k$ if and only if there exists a basis $\B$ of
$H^1(A(\gam),\bF)$ such that $\mu(\GB)\leq k$.
\end{thm}

From the computability of the cohomology ring, (cf.~Section \ref{ss:biautomatic}), we have the following consequence:


\begin{cor}\label{cor:decidable}
From an arbitrary finitely presented group $G=\langle S\mid R\rangle$ such that $G\cong A(\gam)$ for some finite simplicial graph $\gam$,
the value of $\mu(\gam)$ is computable from $\langle S\mid R\rangle$.
\end{cor}

Recall that a graph is \emph{planar} if its geometric realization can be embedded in the plane, and \emph{outerplanar} if it can be embedded
in the plane in such a way that every vertex is adjacent to the unbounded component of the complement. Moreover, a graph is \emph{linklessly embeddable} in $\mathbb{R}^3$
if there is an embedding of the graph in $\mathbb{R}^3$ such that no pair of cycles are linked after being embedded; observe that this property can be thought of as a 3-dimensional analogue of planarity.


Obtaining the following consequence was another motivation for carrying out the present work.

\begin{thm}\label{cor:planar}
Let $P$ be a property of graphs that is characterized by excluding a class of forbidden subgraphs. A finite simplicial graph $\gam$ has property
$P$ if and only if there exists a basis $\B$ of
$H^1(A(\gam),\bF)$ such that $\GB$ has property $P$.

Moreover, let $P$ one of the following graph properties:

\begin{itemize}

\item Emptiness (having no edges).
\item Being a linear forest (union of disjoint paths).
\item Planarity.
\item Outerplanarity.
\item Linkless embeddability.

\end{itemize}

Then a finite simplicial graph $\gam$ has property $P$ if and only if there exists a basis $\B$ of
$H^1(A(\gam),\bF)$ such that $\GB$ has property $P$.
\end{thm}

One must be careful in generalizing Theorem~\ref{cor:planar} to forbidden minors, since the cohomology basis graph does not behave
well under taking minors of $\gam$; see Section~\ref{sec:minors} below. Theorem~\ref{cor:planar} behaves well for properties which are
monotone with respect to taking minors, such as having Colin de Verdi\`{e}re invariant bounded by a fixed integer.
See~\cite{Chartrand} for a discussion of graphs characterized by forbidden subgraphs.

Analogous characterizations of some of the properties enumerated in Theorem~\ref{cor:planar}
for the complement of a given graph are also possible; see Proposition \ref{prop:complement} below.

As was mentioned already, the formulation of the previous theorem makes no reference to any distinguished set of generators of the group $A(\gam)$. Moreover, information about graph properties can be effectively obtained out of any presentation of the associated right-angled Artin group, via the basis cohomology graph associated to that presentation and using $\mathbb{F}_2$-coefficients (see Example \ref{ex2} and the discussion at the end of Section \ref{ss:biautomatic}). This last observation contrasts with the recent results of Gheorghiu in~\cite{gheorghiu23}, at least
from the computational point of view. Indeed, Gheorghiu finds an intrinsic characterization of right-angled Artin groups on planar graphs
for instance, but it is not clear whether his methods are effective.
In this vein, we note that Theorem~\ref{cor:planar}
also furnishes a linear algebraic characterization of planarity of finite simple graphs, in the spirit of Maclane (cf.~\cite{maclane}).

The full extend of information which can be gleaned from the cohomology basis graph is not yet clear.
\begin{quest}
Let $\gam$ and $A(\gam)$ be as above.
\begin{enumerate}
\item
What is
the effect of different fields on the structure of the cohomology basis graph? Note that depending on the characteristic, different edges may
appear or be deleted.
\item
What further data about the defining graph can be extracted from the cohomology basis graph?
For instance, how can the cycles of $\gam$ be read off?
\item
If one considers all possible bases for the cohomology of $A(\gam)$, one obtains a partially ordered set under inclusion of subgraphs. Which
subgraphs between $\gam$ and the complete graph on the vertices of $\gam$ occur? To what extent does
the answer depend on $\gam$?
\end{enumerate}
\end{quest}

The paper is structured as follows: in Section \ref{sect:background} we provide
background about right-angled Artin groups and the Colin de Verdi\`{e}re invariant, as well as different aspects of the
cohomology of right-angled Artin groups that are relevant
to our analysis. Section \ref{sect:null-connected} contains the proofs of the main results. In Section \ref{sect:examples} we
describe some specific examples which illustrate the difficulties in establishing the main result.
We conclude with Section \ref{sect:minors}, where we investigate the relationship between
the minors of the defining graph and of the cohomology basis graphs in more detail.


\section{Background}
\label{sect:background}
\subsection{Notation and terminology}
We follow generally accepted conventions and notation for graphs; see~\cite{diestel-book}, for instance. Of particular interest will
be \emph{graph minors}. A graph $\Lambda$ is an \emph{elementary minor} of $\gam$ if $\Lambda$ is obtained from $\gam$ by deleting
a vertex, deleting an edge, or contracting an edge. We say that $\Lambda$ is a \emph{minor} of $\gam$ if there is a sequence
$\{\gam_0,\ldots,\gam_n\}$ of graphs such that \[\gam=\gam_0,\quad \gam_n=\Lambda,\quad \gam_{i+1}\textrm{ is an elementary minor
of $\gam_i$ for all $i$}.\]

We will adopt some slightly nonstandard linear algebra terminology. For a matrix $A$, we will write entries $a_i^j$, where $i$ indicates the
row and $j$ indicates the column. We will write $\{a_1,\ldots,a_n\}$ for the rows of a matrix.
A \emph{minor} of $A$ is simply a (possibly empty) square submatrix of $A$ obtained by deleting some (possibly empty) collection
of rows and columns of $A$. The \emph{dimension}
of such a minor is just the number of rows or columns in the minor.

We will write $S_n$ for the symmetric group on $n$ letters, and $\sigma$ for an arbitrary element of $S_n$.

\subsection{Cohomology of right-angled Artin groups}\label{ss:coho-raag}

We recall some basic facts about the structure of the cohomology algebra of a right-angled Artin group $A(\gam)$. The
result recorded here is easy and well--known, and follows from
standard methods in geometric topology together with the fact that the Salvetti complex associated to $\gam$ is a classifying space for
$A(\gam)$. More details can be found in in~\cite{FKK2020exp,koberda21survey}, for instance.

Let $V(\gam)=\{v_1,\ldots,v_n\}$ and $E(\gam)=\{e_1,\ldots,e_m\}$ be the vertices and edges of $\gam$,
and write $\smile$ for the cup product pairing
on $H^*(A(\gam),\bF)$.

\begin{lem}\label{lem:coho}
Let $\gam$ be a finite simplicial graph. Then there are bases $\{v_1^*,\ldots,v_n^*\}$ for $H^1(A(\gam),\bF)$ and $\{e_1^*,\ldots,e_m^*\}$ for
$H^2(A(\gam),\bF)$ such that:
\begin{enumerate}
\item
We have $v_i^*\smile v_j^*= 0$ if and only if $\{v_i,v_j\}\notin E(\gam)$;
\item
We have $v_i^*\smile v_j^*=\pm e_k^*$ whenever $\{v_i,v_j\}=e_k\in E(\gam)$.
\end{enumerate}
\end{lem}

Let \[w_i=\sum_{j=1}^n a_i^j v_j^*\] for $i=1,2$, and for coefficients $a_i^j$ in a field.
From Lemma~\ref{lem:coho}, we observe that $w_1\smile w_2\neq 0$ if and only if there is a pair of indices $j$ and $k$ such that $v_j^*\smile
v_k^*\neq 0$ and the matrix \[\begin{pmatrix}a_1^j&a_1^k\\ a_2^j&a_2^k\end{pmatrix}\] is nonsingular. This fact will be used implicitly
throughout the rest of this paper.

\subsection{The Colin de Verdi\`ere invariant}\label{ss:cdv}
The Colin de Verdi\`ere invariant of a graph is an invariant arising from spectral graph theory, and which gives a vast generalization of
classical planarity criteria for graphs.
General references on the Colin de Verdi\`ere invariant are \cite{CdV1990} and \cite{CdV1999}. We give a brief
summary of the definition and main properties of this invariant for the convenience of the reader, and which can be found in the
aforementioned references.

We represent a finite simple connected graph $\gam$ by its vertex set $V=\{1,\ldots,n\}$ and its edge set $E$. We consider symmetric real $n\times n$
matrices $M$ such that the following three conditions hold:
\begin{itemize}
\item
For all distinct indices $1\leq i,j\leq n$, we have $M_i^j<0$ if $\{i,j\}\in E$, and $M_i^j=0$ otherwise;
\item
$M$ has exactly one negative eigenvalue of multiplicity one;
\item
There is no nonzero symmetric real $n\times n$ matrix $X$ such that $MX=0$ and such that $X_i^j=0$ whenever $i=j$ or $M_i^j\neq 0$.
\end{itemize}


The Colin de Verdi\`ere invariant $\mu(\gam)$ is the largest corank of any $M$ satisfying these conditions. Here, for a symmetric matrix $m\times m$ of rank $r$, the \emph{corank} is equal to $m-r$.


Although the original definition of the invariant assumes the graph to be connected, it is easy to extend the definition to disconnected
non-empty graphs by taking the maximum of the value of the invariant on the components; for empty graphs,
the invariant is defined to be zero; see \cite{KLV97}.

The following result illustrates the power of this invariant:

\begin{thm}

Let $\Gamma$ be a finite simple graph such that $\mu(\Gamma)\leq k$. Then:

\begin{itemize}

\item $k=0$ if and only if $\Gamma$ has no edges.

\item $k=1$ if and only if $\Gamma$ is a union of disjoint paths.

\item $k=2$ if and only if $\Gamma$ is outerplanar.

\item $k=3$ if and only if $\Gamma$ is planar.

\item $k=4$ if and only if $\Gamma$ is linklessly embeddable in $\R^3$.

\end{itemize}

\end{thm}

It is also possible to describe the planarity properties of the complement of the graph in terms of this invariant. Recall that two vertices in a graph are \emph{twins}
if they have the same set of neighbours.

\begin{thm}
\label{thm:notwins}
Let $\Gamma$ be a finite simple graph of $n$ vertices with no twin vertices, and such that $\mu(\Gamma)\geq n-k$. Then:

\begin{itemize}

\item $k=5$ if and only if the complement of $\Gamma$ is planar.

\item $k=4$ if and only if the complement of $\Gamma$ is outerplanar.

\end{itemize}

\end{thm}

The previous result can be refined further. For example, the ``only if" implications
are true without assuming any conditions on the vertices. Moreover,
there are also (weaker) characterizations of the complement being a linear forest or linklessly embeddable
in $\mathbb{R}^3$, also using the Colin de Verdi\`{e}re invariant. See \cite{KLV97}.



\subsection{Right-angled Artin groups and formality}
As was noted already, the isomorphism type of $A(\gam)$ determines the isomorphism type of the defining graph $\gam$.
In~\cite{KoberdaGAFA}, a proof was given that passed through the cohomology rings of right-angled Artin groups; cf.~\cite{koberda21survey}.
That is, the cohomology of the right-angled Artin group up to dimension two, together with the cup product pairing, determines the isomorphism
type of the underlying graph $\gam$. This fact fits into a much broader theory of formality, and in particular $1$--formality. This is a phenomenon
closely related to the de Rham fundamental group, quadratic presentability of the Mal'cev algebra,
and the minimal $1$--model; see~\cite{ABCKT}. For general Artin groups, $1$--formality
is a consequence of the work of Kapovich and Millson~\cite{KapMill}. Categorical perspectives on right-angled Artin groups and their
defining graphs are investigated by Grossack~\cite{grossack}.



\subsection{Effectiveness, automaticity, and computation}\label{ss:biautomatic}

For a general finitely presented group, one can algorithmically recover the cohomology ring structure in degree one (i.e. products of elements in degree one),
using a standard $5$--term exact sequence
from group cohomology (arising from the Lyndon--Hochschild--Serre Spectal Sequence; cf.~\cite{Brown-coho}).
Explicitly, if $G$ is a finitely generated group and $H$ is its abelianization, then one can compute the
kernel of the cup product map \[\bigwedge^2 H\longrightarrow H^2(G)\] via the lower central series. In particular, it can be decided which products in $\bigwedge^2$ are trivial or not. Thus, from any generating set for the
first cohomology of a right-angled Artin group, one can recover the structure of the associated cohomology basis graph without appealing to the duals of Artin generators. If we consider coefficients in $\mathbb{F}_2$ and generating sets whose cardinality is the rank of the group, this method and Corollary~\ref{cor:minimal-edge} provide an effective way of reconstructing $\gam$ from the data of
$\haf$, and assuming all queries take a unit amount of time, one can construct a naive algorithm to compute $\gam$ whose complexity is
 $O(e^{|V(\gam)|^2})$. The complexity of recovering $\gam$ from an arbitrary finite presentation of $A(\gam)$ is a somewhat
different matter.


We remark that in general, computing the second cohomology of a finitely presented group is difficult. Indeed, one can
bootstrap the unsolvability of the word problem in general finitely presented groups to prove that it is generally undecidable whether or not
$H^2(G,\Z)=0$; cf.~\cite{gordon-coho}.

For a right-angled Artin group, these pathologies do not occur. Indeed, right-angled Artin groups are \emph{biautomatic};
cf.~\cite{CharneyGD,deligne,brieskorn-saito,charney-automation,charney-biautomatic}.
The exact definition
of this property is irrelevant for our purposes, and we direct the reader to the seminal text~\cite{automatic-group}.

The property of automatic or biautomatic does not depend on the underlying presentation, though passing between different automatic
structures can be mysterious. For a finitely presented group which is known to be biautomatic, practically finding a biautomatic structure
is not always clear.
From a given biautomatic structure on a group $G$, it is a theorem of Bridson--Reeves~\cite{BridsonReeves} that there is
an algorithm to construct a finite dimensional approximation to a classifying space for $G$ (i.e.~a finite dimensional skeleton of $BG$).
Thus, for an arbitrary finitely presented group
which is abstractly isomorphic to a right-angled Artin group, there is an algorithm which computes all of $H^*(G,\Z)$.

Observe that any presentation of a right-angled Artin group defines a basis for the cohomology, and then the arguments above allow us to directly compute the cohomology basis graph.
Any bound on the Colin de Verdi\`{e}re invariant for this graph immediately gives the same bound for the defining graph.
In particular, we obtain a planarity test for the defining graph whose input is any presentation of the right-angled Artin group.
See Example \ref{ex2}.




\section{$\Gamma$--null-connectedness and the proof of the main results}
\label{sect:null-connected}

The ideas we use to establish the main result expand and generalise the constructions developed by the first three authors in~\cite{FKK-hamilton},
which in turn are partially inspired by
classical expansions of the determinant relying on the computation of $2\times 2$ minors, such as Laplace
expansion, the Dodgson condensation formula,
and the Sylvester formula~\cite{MR3208413}.

\subsection{Null-connectedness}
For the remainder of this section, we fix a finite simple graph $\gam$, with vertices $\{1,\ldots,n\}$, as well as $A\in\mathrm{M}_n(\bF)$.
The indices in the labelling of columns of $A$ will be identified with the vertices of $\gam$.

Recall that we write row vectors of $A$ as $a_r=(a_r^1,\ldots,a_r^n)$. We say two rows $a_r$ and $a_s$ of $A$ are \emph{$\gam$--null-connected}
if the minor \[\begin{pmatrix}a_r^i&a_r^j\\ a_s^i&a_s^j\end{pmatrix}\] is singular whenever $\{i,j\}$ is an edge of $\gam$.


A submatrix $M$ of $A$ will be called a \emph{$\gam$--$1$-block} if the following conditions are satisfied.
\begin{enumerate}
\item
$M$ has at least two rows and two columns.
\item
All entries of $M$ are nonzero.
\item
The indices of the columns occurring in $M$ span a connected subgraph of $\gam$.
\item
The rows of $A$ which meet $M$ span a connected graph with the $\gam$--null-connectedness adjacency relation.
\item
$M$ is maximal with respect to these conditions, in the sense that there is no submatrix $N$ of $A$ which properly contains $M$ and which
satisfies the previous conditions.
\end{enumerate}

A \emph{$\gam$--$1$-minor} is a minor of $A$ with at least two rows, and which is contained in a $\gam$--$1$-block.

\begin{lem}\label{lem:row-space}
Let $M$ be a $\gam$--$1$-block in $A$. Then the row space of $M$ is one--dimensional.
\end{lem}

\begin{proof}
Clearly the row space of $M$ is at least one--dimensional. Let $\{i,j\}$ be indices of columns in $M$ that span an edge of $\gam$, and let
$a_1$ and $a_2$ be null-connected rows of $A$. We have that the matrix \[\begin{pmatrix}a_1^i&a_1^j\\ a_2^i&a_2^j\end{pmatrix}\] is singular
and has only nonzero entries, whereby
it follows that the vector $(a_2^i,a_2^j)$ is a nonzero scalar multiple $\lambda_{i,j}$
of $(a_1^i,a_1^j)$. By definition, the indices of the columns meeting $M$ span a
connected subgraph $\Lambda\subset\gam$. For every edge of $\Lambda$ spanned by vertices
$\{s,t\}$, we obtain a nonzero scalar $\lambda_{s,t}$ relating $(a_2^s,a_2^t)$ and $(a_1^s,a_1^t)$. Moreover, if two
edges of $\Lambda$ share a vertex then the two corresponding scalars must coincide.
It follows by induction on the diameter
of the subgraph $\Lambda$ and from the fact
that all entries in $M$ are nonzero that
the scalars $\lambda_{s,t}$ depend only on $a_1$ and $a_2$ and not on the indices $s$ and $t$ of the columns.
It follows that the two rows $a_1$ and $a_2$ of $M$ are scalar multiples of each
other. Since the rows of $M$ span a connected graph under the null-connectivity relation, we see that any two rows of $M$ are scalar
multiples of each other, the desired conclusion.
\end{proof}

The following property of $\gam$--$1$-blocks is crucial for canonical sorting of summands making up the determinant. Again, a similar
statement and proof are found as Lemma 2.10 in~\cite{FKK-hamilton}.

\begin{lem}\label{lem:disjoint}
Let $M_1$ and $M_2$ be two distinct $\gam$--$1$-blocks of $A$. Then $M_1$ and $M_2$ are disjoint as submatrices of $A$.
\end{lem}
\begin{proof}
Let $I\subset \{1,\ldots,n\}$ be the set of columns meeting $M_1$ and $J\subset\{1,\ldots,n\}$ be the set of columns meeting $M_2$,
and let $\{\ell_1,\ldots,\ell_s\}$ be the indices of the rows in $M_2$. Write $\Lambda_I$ and $\Lambda_J$ for the corresponding connected
subgraphs of $\Gamma$.
Suppose $a_{\ell_1}^{k}$ is an entry appearing in both $M_1$ and $M_2$. We will show that in this case $M_1=M_2$.

By definition, for all $i\in I$ we have that $a_{\ell_1}^i$ is a nonzero entry of $M_1$, and that $a_{\ell_m}^k$ is a nonzero entry of $M_2$.
By Lemma~\ref{lem:row-space}, the row spaces of both $M_1$ and $M_2$ are $1$--dimensional.
We have $k\in I\cap J$ and $a_{\ell_1}$ is a row meeting both $M_1$ and $M_2$.

If $a_{\ell_m}$ is another row meeting $M_2$ then $a_{\ell_1}$ and $a_{\ell_m}$ lie in the same $\gam$-null-connected component of the
rows of $A$. Suppose first that they are actually $\gam$-null-connected. Then we have that
the matrix
\[\begin{pmatrix}a_{\ell_1}^{k}&a_{\ell_1}^i\\ a_{\ell_m}^k &a_{\ell_m}^i\end{pmatrix}\] is necessarily singular, and consists of all nonzero
entries since $a_{\ell_m}^k\neq 0$. By the same argument as in Lemma~\ref{lem:row-space} (using the connectivity of
$\Lambda_I$), the rows $a_{\ell_1}^I$ and $a_{\ell_m}^I$,
consisting of entries in the columns indexed by $I$, are nonzero scalar multiples of each other. Since the rows of $M_2$ span a connected
graph under the $\gam$-null-connectivity relation, we obtain that $a_{\ell_1}^I$ and $a_{\ell_m}^I$ are nonzero scalar multiples
of each other for all $1\leq m\leq s$. By the maximality condition on $M_2$ and the connectivity of $\Lambda_I$ and $\Lambda_J$,
we have that $I\subseteq J$. By symmetry, $I=J$.
By the maximality conditions on $M_1$ and $M_2$, we obtain that $M_1=M_2$, the desired conclusion.
\end{proof}

Lemma~\ref{lem:disjoint} allows us to canonically partition the entries of a matrix into three different types:
\begin{enumerate}
\item
Nonzero entries that lie in a $\gam$--$1$-block.
\item
Nonzero entries that do not lie in a $\gam$--$1$-block.
\item
Zero entries.
\end{enumerate}

Next, we need to generalize to this context the notion of \emph{$\gam$--$1$-track} defined in \cite{FKK-hamilton}. By definition, this is a sequence
$\{A_1,\ldots,A_k\}$ of minors of $A$ with the following properties.
\begin{enumerate}
\item
For each $i$, the minor $A_i$ is either a $1\times 1$ submatrix or a $\gam$--$1$-minor.
\item
Each column of $A$ meets exactly one $A_i$.
\item
For $i\neq j$, it is not the case that $A_i$ and $A_j$ belong to a common $\gam$--$1$-minor.
\end{enumerate}

Two $\gam$--$1$-tracks are said to be \emph{different} if they consist of different sets of submatrices of $A$.

The preceding definitions serve to sort the summands that make up the determinant of the matrix $A$, the latter of which is simply viewed as
a signed combination of products of matrix entries. We write $\mathfrak a=(a_{\sigma(1)}^1,\ldots,a_{\sigma(n)}^n)$ for an arbitrary
string of nonzero entries of $A$. Such a string \emph{belongs} to a $\gam$--$1$-track $\{A_1,\ldots,A_k\}$ if for all $i$ there exists a $j$
such that $a_{\sigma(i)}^i$ is an entry of $A_j$.

\begin{lem}\label{lem:unique-track}
Let $\mathfrak a=(a_{\sigma(1)}^1,\ldots,a_{\sigma(n)}^n)$ be a string of nonzero entries of $A$. Then $\mathfrak a$ belongs to a unique
$\gam$--$1$-track in $A$.
\end{lem}

\begin{proof}
For $B$ a $\gam$--$1$-block, we let $\{b_1,\ldots,b_s\}$ be the (possibly empty) set of entries of $\mathfrak a$ which lie in $B$.
Each row and each column of $A$ contains exactly one entry of $\mathfrak a$, and so $\{b_1,\ldots,b_s\}$ defines a $\gam$--$1$-minor $A_B$
in $B$ of dimension exactly $s$. By construction, for distinct $\gam$--$1$-blocks $B_1$ and $B_2$, the $\gam$--$1$-minors $A_{B_1}$ and
$A_{B_2}$ are disjoint. The remaining entries of $\mathfrak a$, say $\{c_1,\ldots,c_t\}$, are $1\times 1$ nonzero matrices that belong to no
$\gam$--$1$-block. Thus, the $\gam$--$1$-track associated to $\mathfrak a$ is \[\{A_B\}_{B\in\mathcal B}\cup \{c_1,\ldots,c_t\},\] where
here $\mathcal B$ ranges over all $\gam$--$1$-blocks of $A$. The disjointness of distinct $\gam$--$1$-blocks guarantees that this is
actually a $\gam$--$1$-track.

It is clear that this $\gam$--$1$-track in $A$ is unique. Indeed, the constituents $\{c_1,\ldots,c_t\}$
and $\{A_B\}_{B\in\mathcal B}$ are canonically defined and hence unique.
\end{proof}

Recall the Leibniz formula for the determinant of an $n\times n$ matrix, namely \[\det A=\sum_{\sigma\in S_n}
\sgn(\sigma)\prod_{i=1}^n a_{\sigma(i)}^i.\]

For a given $\gam$--$1$-track $\T$ of $A$, we write $(\det A)_{\T}$ for the restriction of the sum defining the determinant to permutations
$\sigma$ such that $\mathfrak a=(a_{\sigma(1)}^1,\ldots,a_{\sigma(n)}^n)$ belongs to $\T$.

\begin{lem}\label{lem:zero-det}
Let $\T$ be a $\gam$--$1$-track of $A$.
Suppose $\T$ contains a minor of dimension at least two. Then \[(\det A)_{\T}=0.\]
\end{lem}
\begin{proof}
Let $M$ be such a minor. Without loss of generality, $M$ consists of the top left $k\times k$ minor of $A$, and we may identify $S_k$ with
the group of permutations of the rows of $M$, and we extend these permutations by the identity. If $\mathfrak a$ belongs to $\T$ and
$\tau\in S_k$, then
so does the string $\mathfrak a^{\tau}$, which is obtained by applying $\tau$ to the row indices of $\mathfrak a$. Since the signature
of a permutation is a homomorphism, we have that the contribution of $\mathfrak a^{\tau}$ to $(\det A)_{\T}$ differs from that of $\mathfrak a$
by $\sgn (\tau)$. Since the row space of a $\gam$--$1$-block is one-dimensional, we have that the product of entries of $\mathfrak a$
and $\mathfrak a^{\tau}$ are equal. It is now immediate that $(\det A)_{\T}=0$, since exactly half the permutations in $S_k$
have signature $1$ and
half have signature $-1$.
\end{proof}

In the case of $\bF_2$, Lemma~\ref{lem:zero-det} could be proved by simply noting that if $\T$ contains a minor of dimension at least two
then an even number of distinct strings $\mathfrak a$ belong to $\T$.

\begin{cor}\label{cor:reordering}
Suppose $A$ is invertible. Then there exists a reordering of the rows of $A$ such that for all edges $\{i,j\}$ of $\gam$, we have
that $a_i$ and $a_j$ are not $\gam$--null-connected.
\end{cor}
\begin{proof}
We suppose the contrary and argue that $\det A=0$. Let $\mathfrak a=(a_{\sigma(1)}^1,\ldots,a_{\sigma(n)}^n)$ be a string of
nonzero entries of $A$. By assumption, there is an edge $\{i,j\}$ of $\gam$ which witnesses the fact that $\sigma$ fails to be
a suitable reordering. By reordering the rows by $\sigma^{-1}$, we may assume $\mathfrak a=(a_1^1,\ldots,a_n^n)$. We have that the
matrix \[M=\begin{pmatrix}a_i^i&a_j^i\\ a_i^j&a_j^j\end{pmatrix}\] is singular, because these rows must be $\gam$--null-connected. Since
the matrix $M$ has nonzero diagonal entries and is singular, it must lie inside of a $\gam$--$1$-block of $A$. It follows that the unique
$\gam$--$1$-track $\T$ to which $\mathfrak a$ belongs contains a $\gam$--$1$-minor of dimension at least two, and so $(\det A)_{\T}=0$.
The choice of $\mathfrak a$ was arbitrary, and so each such string belongs to a $\gam$--$1$-track $\T$ such that $(\det A)_{\T}=0$. Now,
the uniqueness of the track to which $\mathfrak a$ belongs implies that $\det A$ is the sum of $(\det A)_{\T}$, where $\T$ varies over
all possible $\gam$--$1$-tracks. The desired result now follows.
\end{proof}

\begin{proof}[Proof of Theorem~\ref{thm:main}]
We fix notation and write
$\{v^*_1,\ldots,v^*_n\}$ for the standard basis for $H^1(A(\gam),\bF)$, and we let $A\in \mathrm{GL}_n(\bF)$ be arbitrary. We let
$\B=\{w_1,\ldots,w_n\}$ be the result of applying $A$, viewed as a change of basis, so that \[w_i=\sum_{j=1}^n a_i^j v_j^*.\] We let $\sigma$
be a reordering of the rows of $A$ as guaranteed by Corollary~\ref{cor:reordering}, and we relabel the vectors $\{w_1,\ldots,w_n\}$
according to $\sigma$. Computing $w_i\smile w_j$, we see that this cup product is zero if and only if the rows $a_i$ and $a_j$ of
$A$ are $\gam$--null-connected. It follows that the cohomology basis graph $\G_{\B}$ contains $\gam$ as a subgraph,
as desired.
\end{proof}

Observe that Corollary \ref{cor:minimal-edge} is immediate from Theorem~\ref{thm:main}, as for an inclusion of graphs
$\Gamma'\subseteq\Gamma$ with the same number of vertices, equality in the number of
edges implies isomorphism. Theorem \ref{thm:cdv} is a consequence of Theorem~\ref{thm:main} and the minor-monotonicity of the Colin de Verdi\`{e}re invariant. As right-angled groups are biautomatic, Corollary \ref{cor:decidable} follows. Finally, Theorem \ref{cor:planar} is implied in turn by Theorem \ref{thm:cdv}, together with the properties of the invariant discussed in Section \ref{ss:cdv}.


We conclude by stating a property concerning complements, which is analogous to Theorem \ref{cor:planar}:

\begin{prop}
\label{prop:complement}

Let $\Gamma$ be a finite simple graph with no twin vertices. Then the complement of $\Gamma$ is planar (resp. outerplanar) if and only if for every basis  $\B$ of
$H^1(A(\gam),\bF)$, the complement of the cohomology basis graph $\GB$ is planar (resp. outerplanar).

\end{prop}

\begin{proof}

We prove the case of planarity; the argument for outerplanarity is analogous.
By Theorem \ref{thm:notwins}, the complement of $\Gamma$ is planar if and only if $\mu(\Gamma)\geq n-5$, where $n=|V(\Gamma)|$. By the minor monotonicity of $\mu$, we see that $\mu(\GB)\geq n-5$
for every cohomology basis graph $\GB$. By Theorem \ref{thm:notwins}, the complement of every such $\GB$ is planar.
The other implication is immediate.
\end{proof}


\subsection{A reformulation in terms of edge ideals}
\label{sect:characterizations}
There are equivalent reformulations of the main Theorem \ref{thm:main} in other contexts, which to our knowledge were
also open questions; this was communicated to the authors by A.~Van Tuyl~\cite{tuyl-personal}.
 Here we discuss a perspective from commutative algebra, and in the next section from graph theory.

One fruitful context for investigating the interplay between combinatorics of graphs and commutative algebra is through \emph{clutters}
and \emph{monomial ideals}, and especially \emph{edge ideals}, which are in turn related to the theory of polyhedral products and
Stanley-Reisner rings; the reader is directed to~\cite{morey-villarreal,Ha-VanTuyl} for background.


We consider $k=\bF_2$, the field with two elements, and a polynomial ring $R=k[x_1,\ldots,x_n]$. We let $I\subset R$ be an ideal generated
by square-free monomials of degree two, and we let $\gam_I$ be the graph having $I$ as its edge ideal. A matrix $A\in\mathrm{GL}_n(k)$
determines a linear change of variables $x_i\mapsto w_i$ that preserves degrees of polynomials.
We let $I'$ be the \emph{complement} of $I$, which is to say that $I'$ is generated by all degree two monomials that do not lie in $I$, and
we let $\overline R=R/I'$.

One can now consider the ideal $J\subset k[w_1,\ldots,w_n]$ generated by products $w_iw_j$, with $i\neq j$, and its image $\overline J\subset
\overline R$. We write $\gam_J$ for the graph with vertices $\{w_1,\ldots,w_n\}$, and with an edge whenever the monomial $w_iw_j$
survives in $\overline J$. These graphs have been extensively studied in the literature, see for example \cite{vantuyl} and references therein.

 The following result can easily be seen to follow from Theorem~\ref{thm:main}:

\begin{cor}
The graph $\gam_I$ is a subgraph of $\gam_J$.
\end{cor}

Conversely, observe that every finite simple graph can be represented by $\gam_I$ for some edge ideal. Moreover, every change of
basis can be effected
by an invertible matrix. Thus, if $\gam_I$ is always a subgraph of $\gam_J$ as above, then Theorem~\ref{thm:main} holds over a field of
characteristic two.

\subsection{A reformulation in classical graph theory}

Let $A\in \mathrm{GL}_n(\bF_2)$, and let $\Gamma$ be a fixed graph on $n$ vertices $\{v_1,\ldots,v_n\}$. We now define a new
finite simple graph with vertex set $\{e_1,\ldots,e_n,w_1,\ldots,w_n\}$. The edge relation is given as follows:

\begin{enumerate}
\item We place an edge between $e_i$ and $e_j$ precisely when there is an edge between $v_i$ and $v_j$.

\item We place an edge between $w_i$ and $e_j$ precisely when the entry $a_i^j$ of $A$ is nonzero.

\item For $i\neq j$, we place a new edge between $w_i$ and $w_j$ if and only if there exist indices $k$ and $\ell$ such that
the following conditions hold:
\begin{enumerate}
\item
There is an edge between $v_k$ and $v_{\ell}$.
\item
There is a path of length three between $w_i$ and $w_j$ that contains the edge between $e_k$ and $e_{\ell}$ induced by the previous condition and the two edges arising from (2) above, and there is no path of length two in the subgraph induced by $\{w_i,w_j,e_k,e_{\ell}\}$.
\end{enumerate}
\end{enumerate}

We write $\Gamma'$ for the graph spanned by $\{w_1,\ldots,w_n\}$, and let $\mathcal B$ be the basis for $H^1(A(\Gamma),\bF_2)$
induced by the rows of $A$; by construction, we may naturally identify elements of $\mathcal B$ with $\{w_1,\ldots,w_n\}$.
It is straightforward to see that $w_i$ and $w_j$ are adjacent in $\Gamma'$ if and only if $w_i$ and $w_j$ are adjacent in
$\Gamma_{\mathcal B}$: indeed, this can be seen from explicitly writing out the cup product in terms of the basis of $H^1(A(\Gamma),\bF_2)$
coming from the vertices of $\Gamma$; cf.~Subsection~\ref{ss:coho-raag}.
Thus, Theorem \ref{thm:main} is equivalent to:

\begin{cor}

The graph $\Gamma$ is a subgraph of $\Gamma'$.

\end{cor}


\section{Some examples}
\label{sect:examples}
In this section we describe explicit examples that we find illustrative for understanding the difficulties that arise in attempts to prove the
main result directly.


One of the basic issues that makes finding a proof of Theorem~\ref{thm:main} nontrivial is the ``global" nature of the assertion it makes.
The result says that from an arbitrary basis for $H^1(A(\gam),\bF)$, one can find an assignment between these arbitrary basis vectors
and vectors in the standard basis which respects the cup product structure. Experiments suggest and careful consideration shows
that there is no
canonical way to realize such a bijection, and this is the reason that inductive strategies do not seem to work; even if one assumes the
existence of such a bijection for a proper subgraph, extending by even one vertex seems technically impossible. These issues
already appear in the following simple example:

\begin{ex}
\label{ex0}

Consider the defining graph $\Gamma$ with vertices $\{v_1,v_2,v_3,v_4\}$ and edges
\[\{v_1,v_2\}, \{v_1,v_3\},\{v_1,v_4\},\] and the basis $\mathcal B=\{w_1,w_2,w_3,w_4\}$ for the first cohomology of $A(\Gamma)$ given by
the change of basis matrix \[A=\begin{pmatrix}1&0&0&0\\ 0&1&0&1\\0&1&0&0\\ 0&0&1&0 \end{pmatrix}.\]
Now, the associated cohomology basis graph $\GB$ has edges \[\{w_1,w_2\}, \{w_1,w_3\}, \{w_1,w_4\},\]
which is isomorphic to $\Gamma$. Hence, the natural assignment $v_i\mapsto w_i$ induces an isomorphism of graphs,
and in particular an inclusion $\Gamma\subseteq\GB$.

Now, we add the edge $\{v_2,v_4\}$ to the defining graph, obtaining a new graph $\Lambda$, and we retain the basis
$\mathcal B$. The cohomology basis graph
$\Lambda_{\mathcal{B}}$ has edges \[\{w_1,w_2\}, \{w_1,w_3\}, \{w_1,w_4\}, \{w_2,w_3\},\] and the previous
assignment $v_i\mapsto w_i$ do not extend to the new graphs. Of course, Theorem \ref{thm:main} ensures
that there exists another suitable assignment for the new graphs, for instance one sending $v_3$ to $w_4$ and $v_4$ to $w_3$,
but this is an \emph{ad hoc} modification that is difficult to make canonical.

\end{ex}


In searching for a formula defining a canonical
bijection between the vertices of the graph and the arbitrary basis vectors, one encounters many
reasonable-sounding linear algebraic claims which end up being false. One might hope, for example, that if $A\in\mathrm{GL}_n(\bF_2)$
then there is a reordering of the rows of $A$ so that all the principal minors are nonsingular; one might hope for this to be the case for just the
$2$--dimensional principal minors, and 
it is not difficult to see how this claim would imply Theorem~\ref{thm:main}. A counterexample to the claim
is given by the following:

\begin{ex}
\label{ex1}
Consider the invertible matrix \[A=\begin{pmatrix}1&1&0&0\\ 1&1&1&1\\0&0&1&0\\ 0&1&1&1 \end{pmatrix}.\] Then, a straightforward analysis
shows that in every reordering of the rows there is at least one $2$--dimensional principal minor which is singular.

\end{ex}

One further example that we give in this section concerns the practicality of the planarity test described at the end of Section
\ref{ss:biautomatic}.

\begin{ex}
\label{ex2}
Consider the group $$G=\langle x_1,x_2,x_3,x_4,x_5\textrm{ } | \textrm{ } x_1x_2x_5x_4=x_2x_5x_4x_1, x_3x_2x_5x_4=x_2x_5x_4x_3, x_4x_5=x_5x_4\rangle.$$ This groups is abstractly isomorphic to a right-angled Artin group on a graph with five vertices.
If we consider the dual generators $$\{x^*_1,x^*_2,x^*_3,x^*_4,x^*_5\}$$ in $H^1(G;\mathbb{F}_2)$, it is not hard to check using the arguments of
Subsection~\ref{ss:biautomatic} that the only non-trivial products of two of these elements are $x_2^*x_i^*$ for every $i\neq 2$ and $x_4^*x_5^*$. Hence, the cohomology graph is a star with one additional
edge, which is planar. Thus, the defining graph is also planar, by Theorem \ref{thm:cdv}.

\end{ex}


\section{Graph minors and the cohomology basis graph}\label{sec:minors}
\label{sect:minors}

In this last section, we investigate the behavior of the cohomology graph $\GB$ under taking
minors of $\gam$. This is particularly relevant in light of Wagner's theorem (i.e.~a graph is planar if and only if it does not
admit $K_5$ or $K_{3,3}$ as a minor), though we will show that the cohomology
basis graph and graph minors do not interact in a sufficiently nice way to make this a viable approach to characterizing planarity. Thus,
we have another justification for Theorem~\ref{thm:main} being the ``correct" approach to understanding planarity through cohomological
means. 



Notice that if $\Lambda$ is an elementary minor of $\gam$ obtained by deleting an edge or vertex then there is a natural inclusion of
$\Lambda$ into $\gam$. If $\Lambda$ is an elementary minor under contracting an edge with vertices $v_1$ and $v_2$, then $\Lambda$
is equipped with a special vertex which we can formally and suggestively label $v_1+v_2$.

Let $\B$ be an arbitrary basis for $H^1(A(\gam),\bF)$.
We begin by writing the elements of $\B$ in terms of the cohomology classes that are dual to
vertices of $\gam$, i.e.~ $\{v^*_1,\ldots,v^*_n\}$.
If $\Lambda$ is an elementary minor of $\gam$, then there is a natural way to obtain
a basis $\B'$ for $H^1(A(\Lambda),\bF)$. Then, we have the following moves:

\begin{enumerate}
\item
If $\Lambda$ is obtained from $\gam$ by deleting an edge of $\gam$ then $\B'=\B$;
\item
If $\Lambda$ is obtained from $\gam$ by deleting the vertex corresponding to the cohomology class
$v^*_i$, then the vectors of $\B'$ are simply the vectors of $\B$ with every occurrence of $v^*_i$ deleted and any repeats or trivial vectors
discarded and, if necessary, one further vector discarded to ensure that $\B'$ is linearly independent;
\item
If $\Lambda$ is obtained from $\gam$ by contracting the edge connecting vertices corresponding to $v^*_i$ and $v^*_j$, then the vectors
of $\B'$ are vectors of $\B$ with $v^*_i$ and $v^*_j$ replaced by $v^*_i+v^*_j$, and with any repeats or trivial vectors discarded
and, if necessary, one further vector discarded to ensure that $\B'$ is linearly independent.
\end{enumerate}

To fix terminology, if $\B$ is an arbitrary basis for $H^1(A(\gam),\bF)$, we will call a basis $\B'$ for $H^1(A(\Lambda),\bF)$ obtained by
one of the previous three moves an \emph{elementary minor basis}.

We first make some remarks about the naturality of the transformation from $\B$ to $\B'$. For edge deletions,
there is little to say. If $\Lambda$ is obtained from $\gam$ by deleting an edge between vertices $v_1$ and $v_2$,
then there is a natural homomorphism
$\phi\colon: A(\Lambda)\to A(\gam)$ which is defined by simply declaring that $v_1$ and $v_2$ commute with each other. In this case,
we get a natural map
\[\phi^*\colon H^1(A(\gam),\bF)\to H^1(A(\Lambda),\bF),\] and $\phi^*(\B)=\B'$.

The case of vertex deletion is similar, using the fact that
the inclusion $\Lambda\to\gam$ induces a map $\psi\colon A(\Lambda)\to A(\gam)$, and consequently a map
\[\psi^*\colon H^1(A(\gam),\bF)\to H^1(A(\Lambda),\bF).\] We then consider $\psi^*(\B)$, which contains a basis of $H^1(A(\Lambda),\bF)$,
and so
we discard a vector if necessary to obtain $\B'$.

Though $\B'$ is obtained from $\B$ in a reasonably natural way, there is still a matter of the choice of which
vector to discard.
Proposition~\ref{prop:minor} below shows that the choice made is absorbed by a choice of minor on the cohomology basis graph side.

The case of edge contraction is slightly different, since there is a natural map \[\chi\colon A(\gam)\to A(\Lambda),\] though not the other way.
Let $v_1$ and $v_2$ be two adjacent vertices of $\gam$ that are identified in $\Lambda$, wherein we call the resulting vertex $v_0$. Write
$\{v_1^*,v_2^*,v_0^*\}$ for the corresponding dual cohomology classes.
Under the natural induced map $H^1(A(\Lambda),\bF)\to H^1(A(\gam),\bF)$, the
cohomology class $v_0^*$ is sent to the cohomology class $v_1^*+v_2^*$. Thus, $H^1(A(\Lambda),\bF)$ is
canonically isomorphic to a subspace
of $H^1(A(\gam),\bF)$ with the identification $v_0^*\mapsto v_1^*+v_2^*$. We let
\[\varphi\colon H^1(A(\gam),\bF)\to H^1(A(\gam),\bF)\] be the
endomorphism which is
the identity on all generators dual to vertices other than $v_1$ and $v_2$, and otherwise $v_1^*,v_2^*\mapsto v_1^*+v_2^*$.
Finally, we let
\[\chi_*\colon H^1(\varphi(A(\gam),\bF))\to H^1(A(\Lambda),\bF)\] be the map that is the identity on all
generators dual to vertices other than $v_1$ and $v_2$, and which sends $v_1^*+v_2^*\mapsto v_0^*$.
Then $\chi_*\circ \varphi(\B)$ contains a basis
for $H^1(A(\Lambda),\bF)$, and so we discard a vector if necessary to obtain $\B'$.

\begin{prop}\label{prop:minor}
Let $\gam$ be a graph, let $\Lambda$ an elementary minor of $\gam$ obtained by either edge deletion or
vertex deletion, let $\B$ be an arbitrary basis for $H^1(A(\gam),\bF)$, and let
$\B'$ be an elementary minor basis for $H^1(A(\Lambda),\bF)$. Then the graph $\G_{\B'}$ is a minor of the graph $\GB$.
\end{prop}

\begin{proof}
We verify the claim for the two moves that we allow.

{\bf Edge deletion.}
In this case, $\B=\B'$. Suppose $b_1,b_2\in\B$ are adjacent to each other in $\GB$, and write these basis vectors as a sum of cohomology
classes dual to the vertices of $\gam$. Abusing notation slightly, there is an edge $\{v_1,v_2\}$ of $\gam$ such
that $v_1$ but not $v_2$ occurs in the expression of $b_1$, and $v_2$ but not $v_1$ occurs in the expression of $b_2$. If the edge
$\{v_1,v_2\}$ persists in $\Lambda$ then $b_1$ and $b_2$ remain adjacent in $\G_{\B'}$. If the edge $\{v_1,v_2\}$ does not persist in
$\Lambda$, then $b_1$ and $b_2$ may or may not remain adjacent in $\G_{\B'}$, contingent on the existence of another edge of $\gam$
that witnesses the continued adjacency of $b_1$ and $b_2$.

If $b_1$ and $b_2$ are nonadjacent in $\GB$, then we wish to argue that these vertices remain nonadjacent in $\G_{\B'}$.
Let $\{v_1,v_2\}$ denote an arbitrary edge of $\gam$. Necessarily, one of
the following three possibilities holds, up to relabeling vertices or basis elements:
\begin{enumerate}
\item
Neither $v_1$ nor $v_2$ occurs in the expressions for $b_1$ and $b_2$;
\item
Both $v_1$ and $v_2$ occur in both expressions for $b_1$ and $b_2$;
\item
The class $v_1$ occurs in the expression for $b_1$ but $v_2$ does not occur in the expression for $b_2$.
\end{enumerate}

Now, there is a pair of vertices $\{v_i,v_j\}$ which span an edge of $\gam$ but such that
$v_i^*\smile v_j^*=0$ in $H^1(A(\Lambda))$, with no other cup products between dual vertex basis vectors being changed.
It follows immediately then that pairs of nonadjacent vertices in $\GB$ remain nonadjacent in $\G_{\B'}$.

{\bf Vertex deletion.}
Retaining notation from above, let $\psi^*(\B)\subset H^1(A(\Lambda))$ be the image of $\B$ under the map induced by the inclusion
$\Lambda\to\gam$, and let $b_1,b_2\in\B$. For $i\in\{1,2\}$, writing $\psi^*(b_i)$ and $b_i$ in terms of the vertex duals,
we simply have that a summand $v$ is deleted from $\psi^*(b_i)$ if it occurs in $b_i$. If $\psi^*(b_1)=\psi^*(b_2)$ then we will begin by
deleting one of them (which is vertex deletion in $\GB$) and then proceed by applying $\psi^*$, which will then yield $\B'$ without
any further deletions.

Suppose that $b_1$ and $b_2$ are adjacent in $\GB$, and that this adjacency is witnessed only by edges in $\gam$ that are incident
to $v$. Then after deleting $v$ from $\gam$, all these edges are severed, in which case $\psi^*(b_1)\smile \psi^*(b_2)=0$. If the adjacency
is witnessed by an edge that is not incident to $v$, then $\psi^*(b_1)\smile \psi^*(b_2)\neq 0$.

Suppose that $b_1$ and $b_2$ are nonadjacent in $\GB$. Then for an arbitrary edge $\{v_1,v_2\}$ of $\gam$, we have the three possibilities
as in the case of edge deletion. It is straightforward to check that the three possibilities persist after applying $\psi^*$, in which case
$\psi^*(b_1)\smile \psi^*(b_2)=0$.

Let $\gam_{\psi^*}$ be the graph obtained by taking vertices to be elements $\psi^*(\B)$, and adjacency to be given by nonvanishing cup
product. Then the preceding argument shows that $\gam_{\psi^*}$ is a minor of $\GB$. The basis $\B'$ is obtained by
(possibly) deleting an element of $\psi^*(\B)$, in which case we see that $\G_{\B'}$ is a minor of $\gam_{\psi^*}$, as desired.
\end{proof}

It is not generally true that if $\Lambda$ is obtained from $\gam$ by edge contraction then $\G_{\B'}$ is a minor of $\GB$. Consider for instance the dumbbell graph $\Gamma$ from Figure \ref{figure: example}. It has $n + m + 1$ edges.

\begin{center}
\includegraphics[scale=1]{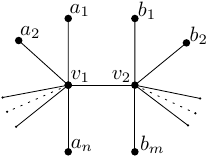}
\captionof{figure}{The dumbbell graph $\Gamma$.}
\label{figure: example}
\end{center}


For compactness of notation, we will conflate names of vertices and corresponding dual cohomology classes, and cease writing asterisk
superscripts.
To simplify notation further, we will write sums of cohomology classes multiplicatively.
Let \[\B=\{b_1,\ldots,b_m,a_1b_1,v_1a_1,\ldots,v_1a_n,
v_2a_1\}.\] It is straightforward to verify that this is indeed a basis for $H^1(A(\gam),\bF)$. The edges
of $\GB$ are of the following form:
\begin{enumerate}
\item
$\{b_i,v_2a_1\}$ for $1\leq i\leq m$;
\item
$\{v_2a_1,v_1a_j\}$ for $1\leq j\leq n$;
\item
$\{v_2a_1,a_1b_1\}$;
\item
$\{v_1a_i,v_1a_j\}$ for $i\neq j$;
\item
$\{a_1b_1,v_1a_j\}$ for $1\leq j\leq n$.
\end{enumerate}

We check easily that $\GB$ has \[\frac{n^2}{2}+\frac{3n}{2}+m+1\]
total edges. Now to compute $\B'$ and $\G_{\B'}$, we introduce the symbol $z$ for $v_1+v_2$.
Each occurrence of $v_1$ and $v_2$ is replaced by $z$. The vertices $v_2a_1$ and $v_1a_1$ become identical, and so we delete one
of them. Then we have \[\B'=\{b_1,\ldots,b_m,a_1b_1,za_1,\ldots,za_n\}.\] The edges of $\G_{\B'}$ are of the following form:

\begin{enumerate}
\item
$\{b_i,za_j\}$ for $1\leq i\leq m$ and $1\leq j\leq n$;
\item
$\{za_i,za_j\}$ for $i\neq j$;
\item
$\{a_1b_1,za_j\}$ for $1\leq j\leq n$.
\end{enumerate}

We check easily that $\G_{\B'}$ has \[\frac{n^2}{2}+mn+\frac{n}{2}\] total edges. Thus, the difference between the total number of edges
of $\G_{\B'}$ and $\GB$ is $(m-1)(n-1)-2$. Evidently, this difference can be made positive by choosing the parameters $n$ and $m$ suitably.
Now, if $\G_{\B'}$ were a minor of $\GB$ then $\G_{\B'}$ would have fewer edges than $\GB$, which is a contradiction.


\section*{Acknowledgments}

\vspace{.5cm}

We thank Y.~Antol\'{i}n, O. C\'ardenas, M. Casals-Ruiz, M. Chudnovsky, J. G\'alvez, M.~Gheorghiu, P. Gim\'enez,
J. Gonz\'alez-Meneses, C. Johnson, A. M\'arquez, J. Mart\'{i}nez, S. Morey, F. Muro, C. Valencia and A. Viruel for helpful comments. We are in particular very grateful to K. Li, M. Tsatsomeros and A. van Tuyl for their interest in our work.
Ram\'{o}n Flores is supported by FEDER-MEC grant PID2020-117971GB-C21 of the Spanish Ministery of Science. Thomas Koberda is partially supported  by NSF Grant DMS-2002596.
Delaram Kahrobaei is supported in part by a PSC-CUNY grant from CUNY.
Corentin Le Coz is supported by the FWO and the F.R.S.–FNRS under the Excellence of Science (EOS) program (project ID 40007542).

The authors are grateful to an anonymous referee who provided numerous helpful comments which greatly improved the paper.
\bibliographystyle{amsplain}

\end{document}